\documentclass[11pt,reqno]{amsart}
\usepackage{amsmath,amsthm,amssymb,enumerate}
\usepackage{hyperref}
\usepackage{mathrsfs}

\numberwithin{equation}{section}

\usepackage{color}
\usepackage{comment}

\newtheorem{theorem}{Theorem}[section]
\newtheorem{lemma}[theorem]{Lemma}
\newtheorem{proposition}[theorem]{Proposition}
\newtheorem{corollary}[theorem]{Corollary}

\theoremstyle{definition}
\newtheorem{assumption}{Assumption}[section]
\newtheorem{definition}{Definition}[section]
\newtheorem{example}{Example}[section]

\theoremstyle{remark}
\newtheorem{remark}{Remark}[section]

\newcommand\bR{\mathbb{R}}
\newcommand\bP{\mathbb{P}}

\newcommand{\bN}{\mathbb{N}}

\newcommand*{\beq}{\begin{equation}}
\newcommand*{\eeq}{\end{equation}}
\newcommand{\E}{\mathbb{E}}

\newcommand{\R}{\mathbb{R}}
\newcommand{\bit}{\begin{itemize}}
\newcommand{\eit}{\end{itemize}}
\newcommand\D{\partial}
\newcommand{\supp}{\text{supp}}
\newcommand{\dist}{\text{dist}}

\begin{document}

\begin{abstract} 
We study densities of two-dimensional diffusion processes with one non-negative component. For such diffusions, 
the density may explode at the boundary, thus making a precise specification of the boundary 
condition in the corresponding forward Kolmogorov equation problematic. 
We overcome this by extending a classical symmetry result for densities of one-dimensional diffusions to our case, 
thereby reducing the study of forward equations with exploding boundary data to the study of a related 
backward equation with non-exploding boundary data.
We also discuss applications of this symmetry for option pricing in stochastic volatility 
models and in stochastic short rate models.
\end{abstract}

\title[Density symmetries for a class of 2-D diffusions]{Density symmetries for a class of 2-D diffusions with applications to finance}
\author[K. Dareiotis and E. Ekstr\"om]{Konstantinos Dareiotis and Erik Ekstr\"om}

\address[K. Dareiotis]{Max Planck Institute for Mathematics in the Sciences, Inselstrasse 22, 04103 Leipzig, Germany}
\email{konstantinos.dareiotis@mis.mpg.de}

\address[E. Ekstr\"om]{Department of Mathematics, Uppsala University, Box 480, 
751 06 Uppsala, Sweden}
\email{Erik.Ekstrom@math.uu.se }

\maketitle

\section{Introduction}

We study the distribution of a special class of diffusions  of the form
\begin{equation}\label{system}
\left\{\begin{array}{ll}
dY_t=\beta_1(Y_t)\,dt + \sigma_1( Y_t) \,dV_t\\
dZ_t=\beta_2(Y_t)\,dt + \sigma_2( Y_t) \,dW_t,\,\end{array}\right.
\end{equation}
where $\beta_i$, $\sigma_i$, $i=1,2$ are given functions, and $V$ and $W$ are two one-dimensional 
Brownian motions. Furthermore, 
the coefficients are specified so that $Y$ is a non-negative process.
This class of decoupled systems includes some common stochastic volatility models (such as the Heston model) 
for derivative pricing,
as well as stochastic short rate models (such as the CIR-model) for derivative pricing.
Denoting by $X=(Y,Z)$, for a given initial condition $X_0= \eta$ with density $\rho \in C^\infty_c((0,\infty) \times \bR)$, the density 
$$p(t,x)=\frac{\mathbb P(X_t\in dx)}{dx}$$
is expected to satisfy the associated forward Kolmogorov  equation
$$\D_tp=L^* p,$$
initial data $p(0,x)=\rho(x)$, where $L^*$ is the formal adjoint of the infinitesimal generator of $X$. However, 
for a characterization of the density in terms of the forward equation, boundary conditions at the spatial 
boundary $\{0\}\times\R$ are needed. Moreover, it is well-known that in many cases of practical importance, 
the density suffers from exploding boundary behaviour,
thus introducing instabilities to any numerical scheme based on discretizing the forward equation. 
To overcome this, one approach
would be to first determine the exact blow-up rate of the density, and then factor out this from the equation to, hopefully, arrive at 
more well-behaved boundary conditions. This, however, requires knowledge about the exact blow-up rate of the density.

Our approach, instead, builds on the extension of a classical symmetry of the transition density for one-dimensional diffusion processes.
In fact, the density 
$$p(t,x_0,x):=\mathbb P(X_t\in dx)/dx$$
of a one-dimensional diffusion $X=X ^{x_0}$ with $X_0=x_0$ satisfies 
\begin{equation}
\label{symmetry1d}
\mu(x_0)p(t,x_0,x)=\mu(x)p(t,x,x_0),
\end{equation}
where $\mu(x)$
is the density of the speed measure (see \cite[Section 4.11]{ItoMcKean}). Along with its theoretical interest, 
this symmetry also has important 
applications for numerical treatments of the density for non-negative processes. Indeed, 
if one seeks the density $p$ of $X_t$, rather than solving the \emph{forward} Kolmogorov  equation in the $x$-variable, one may instead
employ \eqref{symmetry1d} to solve a {\em backward} equation. The advantage of this procedure is in the specification
of boundary conditions, since the density may explode close to the boundary $x=0$, whereas the appropriate
boundary condition of the backward equation is much more well-behaved, compare \cite{ELT} and \cite{ET}.

To the best of our knowledge, extensions of the symmetry relation \eqref{symmetry1d} to higher dimensions
are still missing in the literature. 
In the present article we provide such a symmetry relation for systems of the form \eqref{system} under certain 
conditions on the coefficients, see Theorem~\ref{thm: symmetry}.
Moreover, the stochastic representation appearing in \eqref{q} can typically be characterized as the unique solution of 
a backward equation with well-behaved boundary conditions. For completeness, we also include a study of the associated backward equation. In fact, in Theorem~\ref{thm: solution of PDE} we demonstrate that the stochastic representation
appearing in \eqref{q} can be characterized as the unique solution of 
an associated backward equation for a class of systems that finds applications in mathematical finance. 

The symmetry relation in Theorem~\ref{thm: symmetry} is first proved for processes with the whole plane as state space by approximating the coefficients with smooth coefficients defined on the whole real line. For such problems, the symmetry relation 
\eqref{symmetry2d} is derived using fairly standard methods involving integration by parts, compare Equation \eqref{eq: result for n}.
To pass to the limit, we invoke an approximation result of \cite{BAH} for diffusion processes with H\"older continuous coefficients, see Lemma~\ref{lem: convergence of diffusions}. In our study of the corresponding backward equation, one of the main difficulties is
in specifying the boundary conditions at the plane $y=0$. First, to establish
$C^1$-regularity of the stochastic solution of the equation up to the boundary we again 
approximate the problem with smooth coefficients on the whole plane and then take the limit using appropriate parabolic estimates,
compare Proposition~\ref{pr: continuous derivative x1}. Another key step is to show that the second order terms with at least one derivative in the $y$-direction explode slower than the reciprocal of the corresponding diffusion coefficient, see 
Proposition~\ref{pr: blow up rate second derivative}. This is obtained by using a combination of parabolic estimates and suitable scaling arguments.

Finally let us introduce some notation that will be used throughout the article. Let $T\in (0, \infty)$ and let $(\Omega, \mathscr{F}, \mathbb{F}, \bP)$ be a filtered probability space with the filtration $\mathbb{F}:=(\mathscr{F}_t)_{t \in [0,T]}$ satisfying the usual conditions. On $\Omega$ we consider two $\mathbb{F}$-Wiener processes $(V_t)_{t \in [0,T]}$ and $(W_t)_{t \in [0,T]}$ with correlation $\lambda \in (-1,1)$. For random variables $X, X_n$, $n \in \bN$,  we will write 
$$
X_n \overset{\bP}{\to}X  
$$
if $X_n\to X$  in probability as $n \to \infty$. Let $d$ be a positive integer. For an open set $Q \subset \bR^d$ and an integer $k \in \bN$,  $W^k_2(Q)$ will denote the set of all functions in $L_2(Q)$ having distributional derivatives up to  order $k$  in $L_2(Q)$.  We will denote by $C_b^\infty(Q)$ the set off all smooth real-valued functions on $Q$ that are bounded along with their derivatives of any order. If $Q \subset \bR^d$ is open, we will denote by $C^\infty_c(Q)$ the set of all smooth functions with compact support in $Q$. We also set $\mathbb{W}:=  C([0,T] ; L_2(\bR^2)) \cap L_2([0,T]; W^1_2(\bR^2))$ and  $\mathcal{W}:= C^\infty([0,T] \times \bR^2)\cap \left( \cap_{m=1}^\infty C([0,T]; W^m_2(\bR^2)) \right)$. The notation $(\cdot, \cdot)_{L_2}$ will stand for  the inner product in $L_2(Q)$. If $x \in \bR^2$, then $x_1$ and $x_2$ will denote the first and second coordinates of $x$ with respect to the standard basis in $\bR^2$. Finally, we set $D:= (0,\infty) \times \bR$.

\section{Formulation of the main results}

We consider functions $\beta =(\beta_1, \beta_2)\colon [0,\infty) \to \bR^2 $ and $\sigma=(\sigma_1, \sigma_2) : [0,\infty) \to \bR^2$.
The system 
\begin{equation}
 \label{eq: main SDE}
\left\{\begin{aligned}
       dY_t&=   \beta_1(Y_t) dt + \sigma_1(Y_t) \, dV_t \\
         dZ_t&=\beta_2 (Y_t)dt +\sigma_2(Y_t) \, dW_t,
\end{aligned}\right.
\end{equation}
with initial condition $(Y_0,Z_0)=(\psi,\xi)=\eta$, where $\psi\geq 0$ and $\xi$ are $\mathscr{F}_0$-measurable random variables, will be denoted by $\Pi(\eta;\beta,\sigma)$. We denote by $\lambda\in(-1,1)$ the instantaneous correlation between $V$ and $W$, and we set
\begin{equation*}                                    
 h:= \lambda \frac{\sigma_2}{\sigma_1}
\end{equation*}
and 
\begin{equation}                                      \label{eq: definition a ij}
a_{ij}:= \lambda_{ij} \frac{\sigma_i \sigma_j}{2}
\end{equation}
for $i,j=1,2$,
where $\lambda_{ij}= \lambda$ for $i \neq j$ and $\lambda_{ij}=1$ otherwise. Often, coefficients of SDEs of the type \eqref{eq: main SDE} (say $f$)  will be regarded as functions on subsets of $\bR^2$ by the formula $f(x):= f(x_1)$.

\begin{assumption}              \label{as: definition of class S}
The functions $\beta$ and $\sigma$ satisfy:
\begin{itemize}
\item[(i)] $\beta_i,  \sigma_i \in C([0,\infty))$. Moreover, for every $R >0$ there exists $N_R \in \bR$ such that 

$|\beta_1(r)-\beta_1(r')|+|\sigma_1(r)-\sigma_1(r')|^2  \leq N_R|r-r'|$ 
for all $r,r' \in [0,R]$. 

\item[(ii)]  $\beta_1(0)\geq 0$, $\sigma_1(0)=0$, $\sigma_1(r)>0$ for $r>0$, and $\sigma_2(r) \geq 0$ for $r \geq 0$. Moreover,  there exists a constant $N\in \bR$ such that 
$$ 
|\beta_1(r)|+\sigma_1(r) \leq N(1+r) 
$$
for all $r \geq 0$.

\item[(iii)] 
$h \in C^1((0, \infty))$,   $\beta_1 h, \   a_{11}h' \in C([0,\infty))
$ 
and 
 $
 a_{11}h'(0)=0.
 $
\end{itemize}
In addition,  there exist functions $\sigma^n_1, \ \sigma^n_2 \in C^\infty_b( \bR)$ such that
\begin{itemize}
\item[(iv)] $\sigma^n_i \to \sigma_i $ uniformly on compacts of $(0, \infty)$ as $ n \to \infty$, 

\item[(v)] $\sigma^n_i (r) \geq 1/n$ for all $r \in \bR$,  and there exists a constant $N$ such that 
$$
\sup_n|\sigma^n_1(r)|\leq N(1+|r|)
$$ 
for all $r \in \bR$,

\item[(vi)] $(\lambda \sigma^n_2/\sigma^n_1)' \to  h' $ uniformly on compacts  of $(0, \infty)$  as $n \to \infty$. 
\end{itemize}
\end{assumption}

\begin{remark}
Notice that Assumption \ref{as: definition of class S} is satisfied if for example (i), (ii) and one of the following hold:
\begin{itemize}
\item[1)]  $\lambda=0$, 

\item[2)] $\sigma_2=c \sigma_1$ for some constant $c \in \bR$,

\item[3)] $\sigma_1, \sigma_2 \in  C^1((0,\infty))$, $a_{11} h', \beta_1 h \in C([0,\infty))$ and 
$a_{11} h'(0)=0$.
\end{itemize}
This shows that the Heston model (in which $\sigma_2=c \sigma_1$) is included in the analysis, compare Example~\ref{ex} below.
Similarly, Remark~\ref{rem} discusses derivative pricing models with stochastic interest rate for which $\lambda=0$.

Also notice that under (i) and the linear growth condition from (ii) of Assumption \ref{as: definition of class S}, there exists a unique solution $X:=(Y,Z)$ of $\Pi(\eta;\beta,\sigma)$ (see, e.g., \cite{GK1}). Moreover, due to the assumptions $\sigma_1(0)=0$ and 
$\beta_1(0)\geq 0$, we have $Y_t\geq 0$ for all times $t\in[0,T]$.
\end{remark} 

For the statement of our main theorem, let $\mu \colon (0, \infty) \to \bR$ be given by
\begin{equation}               \label{eq: def mu}
\mu(r)= \frac{1}{a_{11}(r)} \exp\left( \int_1^r  \frac{\beta_1(l)}{a_{11}(l)} dl \right).
\end{equation}
We also introduce the function 
\begin{equation}                        \label{eq: definition of h tilde}
\tilde{\beta}_2:= 2a_{11} h'+2\beta_1h+\beta_2,
\end{equation}
and we set $\tilde{\beta}=(\beta_1, \tilde{\beta}_2)$. By Assumption~\ref{as: definition of class S} we have that $\tilde\beta_2\in C([0,\infty))$.

\begin{theorem}
\label{thm: symmetry}
Let Assumption \ref{as: definition of class S} hold and let $X$ be the unique  solution of $\Pi(\eta;\beta, \sigma)$.  
Assume that $\eta$ has a density $\rho \in C^\infty_c(D)$.
Then for any $g \in C^\infty_c(D)$  we have
\begin{equation}
\label{symmetry2d}
\E \left[ g(X_T) \right] = \int_D g(x) q(T,x) dx,
\end{equation}
where for $x \in D$, 
\begin{equation}
\label{q}
q(T,x):= \mu(x_1) \E  \left[\frac{ \rho(\tilde{X}_T^x)}{\mu(Y^{x_1}_T)}  \right]
\end{equation}
and $\tilde{X}^x=(Y^{x_1},\tilde{Z}^{x})$ is the unique solution of $\Pi(x; \tilde{\beta}, \sigma)$.
Consequently, the restriction of the law of $X_T$ on $D$ has a density given by $q(T, \cdot)$. 
\end{theorem}

\begin{corollary}{\bf (Symmetry of densities.)}\label{cor}
For each $\xi \in D$, let $X^\xi=(Y^{\xi_1}, Z^{\xi})$ and $\tilde{X}^\xi=(Y^{\xi_1}, \tilde{Z}^{\xi})$ denote the unique solutions of $\Pi(\xi; \beta, \sigma)$ and $\Pi(\xi; \tilde{\beta}, \sigma)$, respectively. Suppose that the restriction of the laws of $X_T^\xi$ and $\tilde{X}_T^\xi$ on $D$ have densities $p(T,\xi,x)$ and $\tilde{p}(T,\xi,x)$, respectively, that are continuous in $(\xi,x) \in D\times D$. Then, for all $(\xi,x) \in D \times D$ we have 
$$
\mu(\xi_1)p(T,\xi,x)= \mu(x_1) \tilde{p}(T,x, \xi).
$$
\end{corollary}

Note that
Theorem~\ref{thm: symmetry} transforms the problem of calculating a density with respect to the {\em forward} variables
into a problem of solving a {\em backward} equation for a related process. 
Theorem~\ref{thm: solution of PDE} below provides the 
exact formulation of boundary conditions for backward equations corresponding to diffusions of the form \eqref{eq: main SDE}; for related results, see \cite{BKX} and \cite{ET2}. 

\begin{assumption}         \label{as: 2}
The functions $\sigma_i, \beta_i \colon [0,\infty) \to \bR$ satisfy the following:
\begin{enumerate}  

\item There exists $N \in \bR$ such that 
$$
|\sigma_1(r)|+|\beta_1(r)|\leq N(1+|r|),
$$
for all $r \in [0,\infty)$. \label{as2: linear growth}

\item $\sigma_i(r)>0$ for $r >0$, $\sigma_1(0)=0$, and $\beta_1(0)\geq0$. \label{as2: positivity}

\item $\sigma_i$, $\beta_i \in C^\infty((0,\infty)) \cap C([0,\infty)) $.   \label{as2: continuity}

\item $\beta_i$, $a_{ii} \in C^1([0,\infty))$ with  $\beta_i'$, $a_{22}'$ bounded, and $a_{11}'$ is locally Lipschitz and has linear growth.
 \label{as2: regularity}
 
\item Either $\lambda=0$, or $a_{12} \in C^1([0,\infty))$ and  there exists $N_0 \in (0,\infty)$ such that $\frac{1}{N_0} \sigma_2(r) \leq \sigma_1(r) \leq N_0 \sigma_2(r)$ for all $r$ sufficiently small. \label{as2: diffusion behaviour}

\item It holds that
$$
\int_0^1 \left( \int_r^1 \exp \left(\int_1^s \frac{\beta_1(u)}{a_{11}(u)} \,du \right) \, ds \right) \exp \left(\int_r^1 \frac{\beta_1(s)}{a_{11}(s)} \, ds \right)\frac{1}{a_{11}(r)} \, dr = \infty.
$$
\label{as2: exit}
\end{enumerate}
\end{assumption}

As before, under Assumption \ref{as: 2}, if $Y_0=\psi \geq 0 \ a.s.$, then \eqref{eq: main SDE} has a unique solution $X=(Y,Z)$, and $Y_t \geq 0 \ a.s. $ for all $t \in [0,T]$.
 Let us introduce the differential operator  $L$ given by 
\begin{align}
L \phi(x) :=& \sum_{i,j} a_{ij}(x)\D_{ij} \phi(x)+\sum_i \beta_i(x) \D_{i} \phi(x),
\end{align}
and for a function $g \in C^\infty_c(D)$ let us consider the problem 
\begin{equation} \label{eq: parabolic pde}
  \left\{ \begin{array}{ll}
         \D_t u  + L u=0  & \mbox{in $(0,T)  \times D$}\\
        u(T,x)=g (x)&\mbox{for $x \in D$} \\
        \D_t u+a_{22}\D_{22}u+\sum_i \beta_i\D_iu =0 &\mbox{on $(0,T) \times \D D$}.
        \end{array} \right. 
        \end{equation}

\begin{definition}                            \label{def: solution PDE}
A continuous function  $u: [0,T] \times \overline{D} \to \bR$, will be called a solution of equation \eqref{eq: parabolic pde} if $u \in C^{1,2}((0,T) \times D)$,  $ \D_iu,\D_{22}u, \D_t u \in C((0,T) \times \overline{D})$, and the equalities in \eqref{eq: parabolic pde} are satisfied.
\end{definition}

\begin{theorem}                 \label{thm: solution of PDE}
Let Assumption \ref{as: 2} hold and let $X^x$ be the unique solution of \eqref{eq: main SDE} with initial condition $X_0= x \in \overline{D}$. Then  the function $u(t,x)= \E g(X_{T-t}^x)$ is a solution of \eqref{eq: parabolic pde}.
Moreover, $u$ is the unique solution of equation \eqref{eq: parabolic pde} in the class of functions of at most polynomial growth. 
\end{theorem}

\begin{remark}
Notice that in order to characterize the quantity $\E  \left[{ \rho(\tilde{X}_T^x)}/{\mu(Y^{x_1}_T)}  \right]$ from Theorem \ref{thm: symmetry} as a solution of a parabolic PDE, Theorem \ref{thm: solution of PDE} should be applied with $\tilde{\beta}$ and $\rho/ \mu$ in place of $\beta$ and $g$, respectively.
\end{remark}

\begin{example}\label{ex}{\bf (The Heston stochastic volatility model)}
We illustrate Theorems~\ref{thm: symmetry}-\ref{thm: solution of PDE} by considering the problem of calculating densities in 
stochastic volatility models. For that, assume that a stock price $S$ is modelled by
$$dS_t=\sqrt{Y_t}S_t\,dW_t, \qquad S_0=\zeta,$$
where the instantaneous variance $Y$ is a CIR process given by
$$dY_t=(a-bY_t)\,dt + \sigma \sqrt{ Y_t} \,dV_t, \qquad Y_0= \eta _1.$$
Here $V$ and $W$ are two Brownian motions with correlation $\lambda\in(-1,1)$, 
and $a \geq 0$, $b$ and $\sigma>0$ are constants. Notice that under the assumption that $a\geq 0$, $Y$ stays non-negative but 
may hit zero (if $2a\leq \sigma^2$). In particular, we do not need to impose the usual, more strict, condition $2a> \sigma^2$.
Introducing $Z_t:=\ln S_t$ gives the system
$$\left\{\begin{array}{ll}
dY_t=(a-bY_t)\,dt + \sigma  \sqrt{ Y_t } \,dV_t, \qquad Y_0= \eta _1,\\
dZ_t=-({Y_t}/{2})\,dt + \sqrt{Y_t}\,dW_t, \qquad Z_0 =\eta_2,\end{array} \right.$$
where we assume that $\eta=(\eta_1, \eta_2)$ has a smooth density $\rho$. The density 
\[p(t,x) =\frac{\mathbb P(X_t\in dx)}{dx}\] 
then satisfies the forward equation 
\[\left\{\begin{array}{ll}
\D_t p=L^*p & \mbox{on }(0,\infty)\times D\\
p(0,x)=\rho(x) & \mbox{for }x\in D,\end{array}\right.\]
where 
\[L^*p=\frac{1}{2}\D_{11}(\sigma^2 x_1 p) + \lambda\sigma\D_{12}(x_1 p)+ \frac{x_1}{2}\D_{22} p-\D_1((a-bx_1)p)+\frac{x_1}{2}\D_2 p.\]
To calculate the density using the forward equation, however, is not straightforward since the boundary conditions
at the boundary plane $\{x_1=0\}$ are not known (in fact, the density in the Heston model is known to explode for some parameter regimes, see the classical reference \cite{F}).
Instead, the symmetry relation in Theorem~\ref{thm: symmetry} may be used to translate the forward equation with boundary 
explosion into a backward equation with well-behaved boundary conditions. 

More precisely, let 
\[\mu(x_1)=\frac{2}{\sigma^2 x_1} \exp\left( \int_1^{x_1}  \frac{2(a-bl)}{\sigma^2 l} dl \right),\]
and let $u$ be the unique bounded 
solution (compare Theorem~\ref{thm: solution of PDE}) of the backward equation
\[\left\{\begin{array}{ll}
\D_t u+L u =0 & \mbox{on } (0,\infty)\times D\\
u(T,x)=\frac{\rho(x)}{\mu(x_1)} & \mbox{for }x\in D\\
 \D_t u+ a\D_1 u +\frac{\lambda a}{\sigma}\D_2 u =0 &\mbox{on $(0,T) \times \D D$},\end{array}\right.\]
where 
\[Lu=\frac{\sigma^2 x_1}{2}\D_{11}u + \lambda\sigma x_1\D_{12}u+ \frac{x_1}{2}\D_{22} u+(a-bx_1)\D_1 u+(\frac{2\lambda}{\sigma}(a-bx_1)-\frac{x_1}{2})\D_2 u.\]
Then, by Theorem~\ref{thm: symmetry}, the density is given by 
\[p(t,x)=\mu(x_1)u(T-t,x).\]
\end{example}

\begin{remark}\label{rem}
Another situation in which the above methodology may be useful is in the case of derivative pricing models with stochastic interest
rate. In fact, consider the system
\[\left\{\begin{array}{ll}
dY_t=\beta(Y_t)\,dt + \sigma(Y_t)\,dV_t\\
dZ_t=(Y_t-\frac{\nu^2}{2})\,dt +\nu dW_t,\end{array}\right.\] 
with the interpretation that $Y$ is a stochastic interest rate and $Z$ is the log-price of a risky asset. 
To calculate option prices of the form
\[v=\E\left[\exp\left\{-\int_0^TY_t\,dt\right\}g(Z_{T})\right]\] 
in this model, the density of the process $Z$, killed at the stochastic rate $Y_t$, is needed.
This killed density satisfies a Kolmogorov forward equation; however, if $Y$ is a non-negative diffusion 
(such as in the Cox-Ingersoll-Ross model, see \cite{CIR}), density explosion is expected at the boundary $\{y=0\}$. 
Theorems~\ref{thm: symmetry} and \ref{thm: solution of PDE} can be modified (by adding zero-order terms in the equations) 
in order to cover also the case of derivative pricing models with stochastic interest rates. Note, however, that
the specification of the volatility $\sigma_2=\nu=constant$ suggests that the conditions of 
Assumptions~\ref{as: definition of class S} and \ref{as: 2} are only fulfilled in the case of 
uncorrelated Wiener processes. For ease of presentation, we refrain from including the extension of 
Theorems~\ref{thm: symmetry} and \ref{thm: solution of PDE} to the case of killed processes.
\end{remark}

\section{Proofs of Theorem \ref{thm: symmetry} and Corollary~\ref{cor}}

For the proof of Theorem \ref{thm: symmetry}, we will need the following lemma which is a 
straightforward consequence of \cite[Theorem 2.5]{BAH}.

\begin{lemma}                                 \label{lem: convergence of diffusions}
For $n \in \bN$, let $h^n=(h_1^n,h_2^n)\colon \bR \to \bR^2$ and  $f^n=(f_1^n,f_2^n)\colon \bR \to \bR^2$ be continuous  functions such that:
\begin{itemize}
\item[1)] $h_1^n$ is locally Lipschitz continuous  and $f_1^n$ is  locally $1/2$-H\"older continuous for each $n \in \bN$,
\item[2)] there exists a constant $K$ such that  $ |h_1^n(r)|+|f_1^n(r)| \leq K (1+|r|)$   for all $r \in \bR$ and all $n \in \bN$,
\item[3)] $f_i^n$ and $h_i^n$ converge to $f_i^0$ and $h_i^0$, respectively, uniformly on compact subsets of $\bR$, as $n \to \infty$.
\end{itemize}
Let $(x^n)_{n=0}^\infty \subset \bR^2$, $(t^n)_{n=1}^\infty \subset [0,T]$ be sequences such that $\lim_{n \to \infty}(t^n, x^n)=(t^0,x^0)$, and for each $n \in \bN$, let  $X^n=(Y^n,Z^n)$ be the unique solution of $\Pi(x^n;h^n,f^n)$. 
Then we have the following:
\begin{enumerate}[i]
\item[(i)] It holds that
\begin{equation}                                   \label{eq: convergence in probability}
\sup_{ t \leq T} |X^n_t-X^0_t| \overset{\bP}{\to} 0, \ \text{as $n \to \infty$.}
\end{equation}
\item[(ii)]
Let $g^n, \gamma^n \colon \bR^2 \to \bR$ be continuous functions, bounded and bounded above respectively,  uniformly in $n\in \bN$, such  that  $g^n \to  g^0$ and $\gamma^n \to \gamma^0$ uniformly on compact subsets of $\bR^2$  as $n \to \infty$. Then
\begin{equation}                                     \label{eq: convergence with g}
\lim_{n \to \infty} \E \left[ g^n(X^n_{t^n})e^{ \int_0^{t^n} \gamma^n(X^n_s) \,ds } \right]=  \E \left[ g^0(X^0_{t^0})e^{ \int_0^{t^0} \gamma^0(X^0_s) \,ds } \right].
\end{equation}
\end{enumerate}
\end{lemma}

\begin{proof}
By  \cite[Theorem 2.5]{BAH} we have 
$\lim_{n \to \infty} \E\sup_{t \leq T}|Y^n_t-Y_t^0|^2 =0$,
and consequently, $\sup_{t \leq T}  |Y_t^{n}-Y^0_t| \overset{\bP}{\to} 0$. Moreover, 
for a subsequence we have
$\lim_{k \to \infty} \sup_{t \leq T}|Y_t^{n_k}-Y_t^0|^2 =0$ almost surely.
This, combined with 3) and the uniform continuity of $h_2^0$ and $f_2^0$ on compacts,  imply that almost surely
$$
\lim_{k \to \infty}\sup_{t \leq T} \left(|h^{n_k}_2(Y_t^{n_k}) -h_2^0(Y_t^0) |+|f_2^{n_k}(Y_t^{n_k}) -f_2^0(Y_t^0) | \right)=0.
$$
In particular, almost surely 
$$
\lim_{k\to \infty} \left( \int_0^T |h^{n_k}_2(Y_t^{n_k}) -h_2^0(Y_t^0) | \, dt +\int_0^T |f_2^{n_k}(Y_t^{n_k}) -f_2^0(Y_t^0) |^2 \, dt \right)=0,
$$
which implies (see, e.g., \cite[Theorem 5, p. 181]{KR1})
\begin{align*}
\sup_{t \leq T}  |Z_t^{n_k}-Z^0_t| & \leq \vert z^{n_k}-z^0\vert +\int_0^T |h^{n_k}_2(Y_t^{n_k}) -h_2^0(Y_t^0) | dt  \\
& + \sup_{t \leq T} \left| \int_0^t \left(f^{n_k}_2(Y_s^{n_k}) -f_2^0(Y_s^0) \right) dW_s \right| \overset{\bP}{\to} 0
\end{align*}
as $k \to \infty$. Moreover, for a further subsequence $n_{k_l}$ the convergence takes place almost surely.  Notice that any subsequences $(x^{n_k})_{k \in \mathbb{N}}$,   $(h^{n_k})_{k \in \mathbb{N}}$ and  $(f^{n_k})_{k \in \bN}$ with $n_0=0$ satisfy  the conditions of the lemma, so the convergence above is true along the whole sequence, that is 
$$
\sup_{t \leq T}  |Z_t^{n}-Z^0_t|\overset{\bP}{\to} 0,
$$
as $n \to \infty$,
which proves \eqref{eq: convergence in probability}.
The equality in \eqref{eq: convergence with g} is a direct consequence of \eqref{eq: convergence in probability}.
\end{proof}

For the proof of Theorem \ref{thm: symmetry},  let $\vartheta \in C^\infty(\bR)$ such that $0 \leq \vartheta \leq 1$, $\vartheta(y) =0$  for $y\leq 0$ and  $\vartheta(y) =1$  for $y\geq 1$. 
Also let
$$
\varrho_m(y)= 1/m+\int_{1/m}^y \vartheta(mr) dr.
$$

\begin{remark}                   \label{rem: properties of rho}
Notice that the function $\varrho_m$ has the following properties:
\begin{itemize}
\item[(1)]
There exist $c_m>0$ such that $\varrho_m(r)> c_m $ for all  $r \in \bR$,
\item[(2)]
$\varrho_m(r)=r$ for $r \geq 1/m$,
\item[(3)]
 $0 \leq \varrho_m(r) \leq 1/m$ for $r \leq  1/m$.
\end{itemize}
\end{remark}

\begin{proof}[Proof of Theorem \ref{thm: symmetry}]
Let us extend the coefficients $\beta$ and $\sigma$ on $(-\infty, 0)$ 
by setting them identically equal to their value at $0$.  Let $\beta^n_i \in C_b^\infty(\bR)$   such that there exists a constant $N$ with $\sup_n |\beta^n_1(r)| \leq N (1+|r|)$ for all $r \in \bR$,
 $\beta^n_1(r)=0$ for $|r| \geq 2n$, and 
\begin{equation}                            \label{eq: conv beta n h n}
\beta^n_i \to \beta_i  \ \text{uniformly on compacts of $\bR$} \  \text{as $n \to \infty$}.
\end{equation}
Let $\sigma^n_i$  be approximations of $\sigma_i$  having the properties of Assumption \ref{as: definition of class S}, and let us set 
$$
\sigma^{n,m}_i := \sigma^n_i \circ \varrho_m,  \  \beta^{n,m}_1 := \beta^n_1 \circ \varrho_m, \ \beta^{n,m}_2:= \beta_2^n
$$
and 
$$
a^{n,m}_{ij}=\lambda_{ij}\frac{\sigma^{n,m}_i\sigma^{n,m}_j}{2}, \qquad  \ h^{n,m}= \lambda \frac{\sigma^{n,m}_2}{\sigma^{n,m}_1}.
$$
We introduce the  differential operators,
\begin{align*}
L_{n,m}\phi&= \sum_{ij} a^{n,m}_{ij} \D_{ij} \phi+ \sum_i \beta^{n,m}_i \D_i \phi \\
\tilde{L}_{n,m}\phi&=  \sum_{ij} a^{n,m}_{ij} \D_{ij} \phi+  \sum_i\tilde{\beta}^{n,m}_i \D_i \phi
\end{align*}
 where $\tilde{\beta}^{n,m}_1=\beta^{n,m}_1$ and $\tilde{\beta}^{n,m}_2$ is defined similarly to \eqref{eq: definition of h tilde} with $\beta_i$ and $\sigma_i$  replaced by $\beta^{n,m}_i$ and $\sigma^{n,m}_i$ respectively. 
 We   consider the equation

\begin{equation} \label{eq: PDE truncated}
  \left\{\begin{array}{ll}
     \D_t v(t,x)&= \tilde{L}_{n,m} v(t,x)     \ \ \ \  \text{ on $ (0,T) \times \bR^2$} \\
         \ v(0,x)&=\frac{\rho(x)}{\mu_{n,m}(x_1)}  \ \ \ \  \ \ \  \ \   \text{on $\bR^2$},
        \end{array} \right.
\end{equation}
where for $r \in \bR$
$$
\mu_{n,m}(r) :=\frac{1}{a^{n,m}_{11}(r)} \exp \left( \int_1^r  \frac{\beta^{n,m}_1(l)}{a^{n,m}_{11}(l)} \, dl \right).
$$
Notice that $\rho/ \mu_{n,m} \in C_c^\infty(\bR^2)$ and that $\tilde{L}_{n,m}$ is strongly elliptic (due to (v) of Assumption \ref{as: definition of class S}), with coefficients of class $C^\infty_b(\bR^2)$. Therefore, equation \eqref{eq: PDE truncated} has a unique solution $v_{n,m} \in \mathbb{W}$, which moreover belongs to $\mathcal{W}$. By the Feynman-Kac formula  we have
\begin{equation}              \label{eq: representation of v n m}
v_{n,m}(t,x) = \E \left[ \frac{\rho(\tilde{X} ^{x;n,m}_t)}{\mu_{n,m}(Y^{x_1;n,m}_t)
} \right],
\end{equation}
where we have denoted by $\tilde{X} ^{x;n,m}=(Y^{x_1;n,m}, \tilde{Z}^{x;n,m})$  the unique solution of $\Pi(x;\tilde{\beta}^{n,m}, \sigma^{n,m})$. Let us now set
$$
q_{n,m}(t,x)= \mu_{n,m}(x_1) v_{n,m}(t,x). 
$$
Notice that since $\sigma^{n,m}_1 \in C^\infty_b (\bR)$, $\sigma^{n,m}_1 \geq 1/n$, $\beta_1^{n,m} \in C^\infty_b ( \bR)$  and  $ \beta^n_1=0$  for  $|r| \geq 2n$, 
we have that $\mu_{n,m}\in C^\infty_b(\bR^2)$ and therefore  $q_{n,m} \in \mathcal{W}$. It is easily seen that $q_{n,m}$ is the unique (in $\mathbb{W}$) solution of 
\begin{equation} 
  \left\{ \begin{array}{ll}
     \D_t v(t,x)&= L^*_{n,m} v(t,x)     \ \ \ \  \text{ on $ (0,T) \times \bR^2$} \\
         \ v(0,x)&=\rho(x)  \ \  \qquad \ \   \ \ \ \   \text{on $\bR^2$},
        \end{array} \right.
\end{equation}
where 
\begin{align*}
L^*_{n,m}\phi=\sum_{ij}\D_{ij}(a^{n,m}_{ij} \phi)-\sum_i\D_i (\beta^{n,m}_i \phi).
\end{align*}
For $g \in C^\infty_c((0,\infty)\times \bR)$, the problem
\begin{equation}
  \left\{ \begin{array}{ll}
     \D_t u(t,x) + L_{n,m} u(t,x)=0     &  \text{on $ (0,T) \times \bR^2$} \\
         \ u(T,x)=g(x)  &  \text{on $\bR^2$}
        \end{array} \right.
\end{equation}
has a unique solution $u_{n,m} \in \mathbb{W}$, for which also holds that $u_{n,m} \in \mathcal{W}$.
By the Feynman-Kac formula we have 
\begin{equation}         \label{eq: representation of u n m }
u_{n,m}(t,x)= \E \left[ g(X_{T-t}^{x;n,m}) \right],
\end{equation}
where by $X^{x;n,m}=(Y^{x_1;n,m}, Z^{x;n,m})$ we have denoted the unique solution of 
$\Pi(x;\beta^{n,m}, \sigma^{n,m})$. 
By the \^Ito formula for $\| \cdot \|^2_{L_2}$  (see, e.g.,  \cite{KR}), and the polarization identity $4ab=(a+b)^2-(a-b)^2$ we get
\begin{align*}
(u_{n,m}(T),q_{n,m}(T))_{L_2}&=(u_{n,m}(0), q_{n,m}(0)) _{L_2} 
 +\int_0^T (L_{n,m}^*q_{n,m}(t),u_{n,m}(t))_{L_2} \ dt \\
& -\int_0^T (L_{n,m}u_{n,m}(t), q_{n,m}(t))_{L_2}\,  dt, 
\end{align*}
and since 
$(L_{n,m}^*q_{n,m}(t),u_{n,m}(t))_{L_2}=(L_{n,m}u_{n,m}(t), q_{n,m}(t))_{L_2}$  for all $t \in [0,T]$,  
we obtain by virtue of \eqref{eq: representation of u n m } that
\begin{equation}                    \label{eq: result for n}
\int_{\bR^2} g(x) q_{n,m}(T,x) dx  = \int_{\bR^2} \E \left[ g(X^{x;n,m}_T) \right]\rho(x) dx.
\end{equation}
We want to let $n \to \infty$ in the above relation. Let us  set
$$
\sigma_i^{\infty,m}:= \sigma_i \circ \varrho_m,  \  \beta^{\infty,m}_1:=\beta_1 \circ \varrho_m, \   \beta^{\infty,m}_2:= \beta_2,
$$
 and we denote by   $X^{x;m}=(Y^{x_1;m}, Z^{x;m})$ and  $\tilde{X}^{x;m}=(Y^{x_1;m}, \tilde{Z}^{x;m})$  the unique solutions of $\Pi(x; \beta^{\infty,m},\sigma^{\infty,m})$ and $\Pi(x; \tilde{\beta}^{\infty,m},\sigma^{\infty,m})$ respectively, where $\tilde{\beta}^{\infty,m}_1=\beta^{\infty,m}_1$, and 
 
$$
 \tilde{\beta}_2^{\infty,m}:= |\sigma_1^{\infty,m}|^2\left(\lambda \frac{ \sigma_2^{\infty,m} }{\sigma_1^{\infty,m} } \right)'+2\lambda \frac{\beta^{\infty,m}_1 \sigma_2^{\infty,m} }{\sigma_1^{\infty,m} }+\beta_2.
  $$
Notice that since $\sigma^n_i \to  \sigma_i$,  and  $\beta^n_i \to \beta_i$,  uniformly on compacts of $(0,\infty)$ and $\bR$ respectively, we have that
\begin{equation}          \label{eq: convergence betai sigmai n}
\sigma^{n,m}_i \to \sigma_i^{\infty,m}    \ \text{and}    \ \beta^{n,m}_i \to \beta_i^{\infty,m}  ,  
\end{equation}
uniformly on compacts of  $\bR$  as $n  \to \infty$.  Moreover, by (v) of Assumption \ref{as: definition of class S} and the properties of $\varrho_m$, there exists a constant $N$ such that
$\sup_n \left(|\beta^n_1(r)| +|\sigma^{n,m}_1(r)|\right) \leq N (1+|r|)$ for any $r \in \bR$.
Therefore, by Lemma \ref{lem: convergence of diffusions} combined with the fact that $\rho$ is compactly supported, we get that
 \begin{align*}
  \lim_{n \to \infty}  \int_{\bR^2}  \E \left[ g(X^{x;n, m}_T)\right] \rho(x) dx 
= \int_{\bR^2}  \E \left[ g(X^{x;m}_T)\right] \rho(x) dx.
 \end{align*} 
 For the left hand side of \eqref{eq: result for n} we proceed as follows. Let us set 
 $$
\mu_m(r)= \frac{2}{|\sigma_1\circ \varrho_m(r)|^2} \exp\left( \int_1^r  \frac{2\beta_1\circ\varrho_m(l)}{|\sigma_1 \circ\varrho_m(l)|^2} dl \right).
 $$
Let $K$ be a compact subset of $(0,\infty)$ and set $K_m:= \varrho_m(K)$ which is also a compact set of $(0, \infty)$. We have 
\begin{align}      
\sup_{r \in K} \left| \frac{1}{|\sigma_1\circ \varrho_m(r)|^2}-\frac{1}{|\sigma^n_1\circ \varrho_m(r)|^2}  \right|= \sup_{r \in K_m} \left| \frac{|\sigma_1(r)|^2-|\sigma^n_1(r)|^2}{|\sigma_1(r)|^2|\sigma^n_1(r)|^2} \right|
\end{align}
By the strict positivity of $\sigma_1$ on $(0, \infty)$ and the uniform convergence $\sigma^n_1 \to \sigma_1$ on the compacts of $(0,\infty)$, there exist $c>0$ such that for all $n$ large enough it holds that $\inf_{ r \in K_m}|\sigma_1(r)|^2|\sigma^n_1(r)|^2 >c$. Consequently,
\begin{equation}              \label{eq:mn1}
\lim_{n \to \infty} \sup_{r \in K} \left| \frac{1}{|\sigma_1\circ \varrho_m(r)|^2}-\frac{1}{|\sigma^n_1\circ \varrho_m(r)|^2}  \right|= 0.
\end{equation}
This combined with the uniform convergence  $\beta^n_1 \to \beta_1$  on the compacts of  $\bR$ gives
\begin{equation}        \label{eq:mn2}
\lim_{n \to \infty} \sup_{r \in K} \left| \frac{2\beta_1\circ\varrho_m(l)}{|\sigma_1 \circ\varrho_m(l)|^2}-\frac{2\beta^n_1\circ\varrho_m(l)}{|\sigma^n_1 \circ\varrho_m(l)|^2}\right| =0.
\end{equation}
By \eqref{eq:mn1} and \eqref{eq:mn2} it follows that 
$\mu_{n,m}  \to \mu_{m}$  uniformly on compacts of $(0, \infty)$,  as $n \to \infty$. Similarly, one can easily see that $1/\mu_{n,m} \to 1/\mu_m$, uniformly on compacts of $(0, \infty)$,  as $n \to \infty$.
 Since $\rho$ is compactly supported in $(0, \infty) \times \bR$, we have that 
\begin{equation}                     \label{eq: uniform conv of mu n m}
\lim_{n \to \infty} \| \rho / \mu_{n,m} -\rho/ \mu_{m}\|_{L_\infty(\bR^2)} =0.
\end{equation}
Moreover, by the strict positivity of $\sigma_1$ on $(0,\infty)$  and (iv) of Assumption \ref{as: definition of class S}, we have that $\sigma_2^n/\sigma^n_1 \to \sigma_2/\sigma_1$ uniformly on compacts of $(0,\infty)$, which implies that 
\begin{equation}             \label{eq: con f/sigma n m}
\frac{\beta^{n,m}_1\sigma_2^{n,m}}{\sigma_1^{n,m}} \to \frac{\beta_1^{\infty,m}\sigma_2^{\infty,m}}{\sigma_1^{\infty,m}},
\end{equation}
uniformly on compacts of $\bR$.
In addition, by (iv) and (vi) of Assumption \ref{as: definition of class S} and the properties of $\varrho_m$ we have that
\begin{equation}           \label{eq: con s2 f/s ' n m}
|\sigma_1^{n,m}|^2\left(\lambda \frac{\sigma_2^{n,m}}{\sigma_1^{n,m}}\right)'=|\sigma_1^{n,m}|^2\left( \left(\lambda \frac{\sigma_2^n}{\sigma^n_1}\right)'\circ \varrho_m \right)  \varrho_m' \to |\sigma_1^{\infty,m} |^2\left(\lambda \frac{ \sigma_2^{\infty,m} }{\sigma_1^{\infty,m} }\right)'
\end{equation}
uniformly on compacts of $\bR$.
Thus, from \eqref{eq: con f/sigma n m}, \eqref{eq: con s2 f/s ' n m}, and the properties of $\beta^n_2$, we get that $\tilde{\beta}_2^{n,m} \to \tilde{\beta}_2^m$ uniformly on compacts of $\bR$.
This combined with \eqref{eq: convergence betai sigmai n} and  \eqref{eq: uniform conv of mu n m},   imply by virtue of Lemma \ref{lem: convergence of diffusions}  that for any $x \in \bR^2$
$$
\lim_{n \to \infty} \E \left[ \frac{\rho(\tilde{X}^{x;n,m}_T)}{ \mu_{n,m}(Y^{x_1;n,m}_T)} \right]=\E \left[ \frac{\rho(\tilde{X}^{x;m}_T)}{ \mu_m(Y^{x_1;m}_T)}  \right]
$$
Notice that  $\rho / \mu_{n,m}$ are bounded uniformly in $n\in \bN$, $\mu_{n,m}  \to \mu_{m} $ uniformly on compacts of $(0, \infty)$,  and $g\in C^\infty_c( (0, \infty) \times \bR)$. Consequently, we obtain by Lebesgue's theorem on dominated convergence that
$$
\lim_{n \to \infty} \int_{\bR^2} g(x) q_{n,m}(T,x) dx =\int_{\bR^2} g(x) q_m(T,x) dx,
$$
where 
$$
q_m(T,x):= \mu_m(x_1)\E \left[ \frac{\rho(\tilde{X}^{x;m}_T)}{ \mu_m(Y^{x_1;m}_T)}\right].
$$
Hence,
\begin{equation}                    \label{eq: result for m}
\int_{\bR^2}  \E \left[ g(X^{x;m}_T) \right] \rho(x) dx= \int_{\bR^2} g(x) q_{m}(T,x) dx.
\end{equation}
Now we want to let $m \to \infty$.
One can easily see that $\sigma_i^{\infty,m}   \to \sigma_i$ and  $\beta_i^{\infty,m}   \to \beta_i$ uniformly on compacts of $\bR$,
 which combined with 
  the properties   of $\beta_1$, $\sigma_1$, and $\varrho_m$,   by virtue of Lemma \ref{lem: convergence of diffusions} together with the boundedness of $g$ and the fact that $\rho$ has compact support imply that
\begin{equation}                               \label{eq: RHS limit}
\lim_{m \to \infty} \int_{\bR^2}  \E \left[ g(X^{x;m}_T)  \right] \rho(x) \, dx =\int_{\bR^2}  \E \left[ g(X^x_T) \right] \rho(x) \, dx
\end{equation}
Notice that 
$$
|\sigma_1^{\infty,m}|^2\left(\lambda \frac{ \sigma_2^{\infty,m} }{\sigma_1^{\infty,m}}\right)'= |\sigma_1|^2\left( \left( \lambda\frac{\sigma_2}{\sigma_1} \right)' \circ \varrho_m \right) \  \varrho_m'.
$$
By (iii) of Assumption \ref{as: definition of class S} and (3) of  Remark \ref{rem: properties of rho}  we have
$$
\sup_{r \in \bR} \left||\sigma_1|^2 \left(\left( \lambda  \sigma_2 /\sigma_1 \right)' \circ \varrho_m \right) \  \varrho_m' (r)-|\sigma_1|^2 \left( \lambda{\sigma_2}/{\sigma_1} \right)' (r) \right| 
$$
\begin{align*}
&= \sup_{r \in [0,1/m]} \left| |\sigma_1|^2 \left( \left(\lambda{\sigma_2}/ {\sigma_1} \right)' \circ \varrho_m \right) \  \varrho_m' (r)-|\sigma_1|^2 \left( \lambda{\sigma_2}/{\sigma_1} \right)' (r) \right| 
\\
&\leq \sup_{r \in [0,1/m]} \left||\sigma_1|^2 \left( \lambda{\sigma_2}/{\sigma_1} \right)' \circ \varrho_m  (r)\right|+\sup_{r \in [0,1/m]} \left||\sigma_1|^2 \left( \lambda {\sigma_2}/{\sigma_1} \right)' (r) \right|  \to 0,
\end{align*}
as $m \to \infty$. In addition,  by (iii) of Assumption \ref{as: definition of class S}, we have that
$$
\lambda \beta^{\infty,m}_1 \sigma_2^{\infty,m} /\sigma_1^{\infty,m} \to \lambda \beta_1 \sigma_2/\sigma_1
$$
uniformly on compacts of $\bR$. Consequently, $\tilde{\beta}^m_2 \to \tilde{\beta}_2$ uniformly on compacts of $\bR$.
As before one can easily check that 
$\mu_m  \to \mu$  and  $1/\mu_m \to  1/\mu$ uniformly on compacts of $(0, \infty)$, 
and that 
$$
\lim_{n \to \infty} \| \rho / \mu_{m} -\rho / \mu\|_{L_\infty(\bR^2)} =0.
$$
Putting these facts together implies by virtue of Lemma \ref{lem: convergence of diffusions} that
$$
\lim_{m \to \infty} \int_{\bR^2} g(x) q_{m}(T,x) dx =\int_{\bR^2} g(x) q(T,x) dx,
$$
which combined with \eqref{eq: RHS limit} brings the proof to an end.
\end{proof}

\begin{proof}[Proof of Corollary~\ref{cor}]
Let us fix $(\xi, x) \in D \times D$. Without loss of generality we can assume that on $\Omega$ there exist $\mathscr{F}_0$-measurable random variables $\eta^n=(\eta^n_1, \eta^n_2)$, $n \in \bN$,  having density $\rho^n(\xi-\cdot)$, where $\rho^n(\zeta)= n^2 \rho (n \zeta)$ for a smooth mollifier $\rho$ supported in the unit ball of $\bR^2$. Let $X^{(n)}=(Y^{(n)}, Z^{(n)})$ be the unique solution of $\Pi(\eta^n; \beta, \sigma)$. Notice that we have almost surely 
$$
 \xi_1- \frac{1}{n} \leq \eta^n_1 \leq \xi_1+ \frac{1}{n}.
$$
Consequently, by a comparison principle (see, e.g., \cite[pp.292]{MR2160585}) we have almost surely for all $t \in [0,T]$
$$
Y^{ \xi_1- \frac{1}{n}}_t \leq Y^{(n)}_t \leq Y^{ \xi_1+ \frac{1}{n}}_t.
$$
By \cite[Theorem 2.5]{BAH} we have that 
$$
\lim_{n \to \infty} \E \sup_{t \leq T} | Y_t^{\xi_1 \pm \frac{1}{n}}-Y^{\xi_1}_t| = 0,
$$
which combined with the above inequality gives
$$
\lim_{n \to \infty} \E \sup_{t \leq T} | Y_t^{(n)}-Y^{\xi_1}_t| = 0.
$$
Then one can easily see (as in the proof of Lemma \ref{lem: convergence of diffusions}) that this implies that $\sup_{t \leq T} | Z_t^{(n)}-Z^\xi_t| \overset{\bP}{\to} 0$, as $n \to \infty$. Consequently, for  $g \in C^\infty_c(D)$ we have
\begin{equation}   \label{eq: convergence expectation}
\lim_{n \to \infty}\E g(X^{(n)}_T)= \E g(X^\xi_T).
\end{equation}

On the other hand, we have
\begin{align*}
\int_D g(x) \mu(x) \E  \left[\frac{ \rho^n(\xi-\tilde{X}_T^x)}{\mu(Y^{x_1}_T)}  \right] \, dx = \int_D g(x) \mu(x) \left( \int_{\bR^2}\rho^n(\xi-\zeta) \frac{\tilde{p}(T,x,\zeta)}{\mu(\zeta)} \,d \zeta \right)  \, dx.
\end{align*}
Next notice that for each $x\in D$
$$
\lim_{n \to \infty} \int_{\bR^2} \rho^n(\xi-\zeta) \frac{\tilde{p}(T,x,\zeta)}{\mu(\zeta)} \,d \zeta = \frac{\tilde{p}(T,x,\xi)}{\mu(\xi)},
$$
and for all $n \geq n_0$ and $x\in\supp (g)$
\begin{eqnarray*}
&& \hspace{-10mm}
\left|  g(x) \mu(x) \int_{\bR^2} \rho^n(\xi-\zeta) \frac{\tilde{p}(T,x,\zeta)}{\mu(\zeta)} \,d \zeta \right|\\
 &\leq& \sup_{x \in \supp(g)}  \sup_{\zeta \in B_{1/{n_0}}(\xi)} \left|\frac{\tilde{p}(T,x,\zeta)}{\mu(\zeta)} \right|| g(x) \mu(x)| < \infty,
\end{eqnarray*}
where $n_0$ is such that $B_{1/{n_0}}(\xi)$ (the ball of radius $1/n_0$ centered at $\xi$) is compactly supported in $D$.  Lebesgue's theorem gives
 $$
 \lim_{n \to \infty} \int_D g(x) \mu(x) \E  \left[\frac{ \rho^n(\xi-\tilde{X}_T^x)}{\mu(Y^{x_1}_T)}  \right] \, dx  = \int_D g(x) \mu(x) \frac{\tilde{p}(T,x,\xi)}{\mu(\xi)} \, dx.
 $$
 This, combined with \eqref{eq: convergence expectation} and Theorem \ref{thm: symmetry} imply that
 $$
  \E g(X^\xi_T) =\int_D g(x) \mu(x) \frac{\tilde{p}(T,x,\xi)}{\mu(\xi)} \, dx.
 $$
 Since $g$ was arbitrary, the claim follows.
\end{proof}

\section{Proof of Theorem \ref{thm: solution of PDE}}

The next proposition is an obvious consequence of Lemma \ref{lem: convergence of diffusions}. 
\begin{proposition}              \label{prop: continuity u}
Under the assumption of Theorem \ref{thm: solution of PDE},  the function $u$ is continuous on $[0,T] \times \overline{D}$.
\end{proposition}

\begin{proposition}                    \label{pr: C 1 2}
Under the assumption of Theorem \ref{thm: solution of PDE}, the function $u$ belongs to $C^{1,2}((0,T) \times D)$ and satisfies 
$\D_t u+Lu=0$ for all  $(t,x) \in  (0,T) \times D$.
\end{proposition}

\begin{proof}
Let $(t,x) \in(0,T)\times D$, let $Q \subset \subset D$ be an open rectangle containing $x$,  and set  $\mathcal{R}=(0,T) \times Q$.
The problem 
\begin{equation} \label{eq:parabolic pde}
  \left\{ \begin{array}{ll}
         \D_t f  + L f=0  & \mbox{in $\mathcal{R}$};\\

        f=u &\mbox{on $\D_p \mathcal{R}$}
        \end{array} \right. 
        \end{equation}
        has a unique classical solution (since $L$ has smooth coefficients and is strongly elliptic in $\mathcal{R}$). Let $X^{x,t}$ be the solution of \eqref{eq: main SDE} starting from $x$ at time $t$.  For $\varepsilon>0$, set
\begin{align*}                          
\tau^\varepsilon &= \inf\{ s \geq t \ |\    (s,X^{x,t}_s) \notin \mathcal{R}_\varepsilon\}, \\
\tau &= \inf\{ s \geq t \ |\    (s,X^{x,t}_s) \notin \mathcal{R}\},
\end{align*}
where $\mathcal{R}_\varepsilon:= \{ (s,y) \in \mathcal{R}\ | \ \dist ((s,y), \D \mathcal{R})> \varepsilon\}$. By Ito's formula we have that the process  $(f(s\wedge \tau^\varepsilon , X_{s\wedge \tau^\varepsilon}^{x,t}))_{s \geq t}$,  is a local martingale and bounded (since $f$ is bounded), hence a martingale. Thus, for any $\varepsilon>0$, $s\geq t$,
 $$
 f(t,x)= \E f(s\wedge \tau^\varepsilon , X_{s\wedge \tau^\varepsilon}^{x,t}), 
 $$
 which by letting $\varepsilon \downarrow 0$, by virtue of the continuity  of $f$ up to the parabolic boundary and due to the fact that $\tau_\varepsilon \uparrow \tau$, gives
 $$
 f(t,x) = \E f(s\wedge \tau, X^{x,t}_{s\wedge \tau} ).
 $$
 Choosing $s=T$ in the above equality gives 
 $$
 f(t,x)= \E f( \tau, X^{x,t}_{\tau} )=\E u( \tau, X^{x,t}_{\tau} )=\E g(X^{x,t}_T)=u(t,x)
 $$
where the second equality follows from the fact that $f=u$ on $\D_p\mathcal{R}$ and the third equality follows from the strong Markov property. As $(t,x) \in (0,T) \times D$ and $Q \subset \subset D$ were arbitrary, this brings the proof to an end. 
\end{proof}

\begin{proposition}                              \label{pr: representation derivatives}
Under Assumption \ref{as: 2}, we have $\D_2u(t,x) = \E(\D_2 g) (X^x_{T-t})$ and 
$ \D_{22}u(t,x) = \E( \D_{22}g) (X^x_{T-t})$. In particular, $\D_2 u, \D_{22}u \in C([0,T] \times \overline{D})$. 
\end{proposition}

\begin{proof}
The result follows immediately from straightforward differentiation, from the fact that $X^x=(Y^{x_1}, Z^{x})$,  where 
$Z^{x}= x_2+f(Y^{x_1})$
for some functional $f$, combined with the fact that $g \in C_c^\infty((0, \infty) \times \bR)$. 
\end{proof}

We proceed with the continuity of $\D_1 u$ up to the boundary  $\D D$. If we formally differentiate the equation
$\D_tu + Lu =0$ with respect to $x_1$, we obtain
$$
\D_t (\D_1 u)+\hat{L}(\D_1 u)+f =0,
$$
where the operator $\hat{L}$ is given by 
\begin{equation}                  \label{eq: operator hat}
\begin{aligned}
\hat{L} \phi := \sum_{i,j} a_{ij}\D_{ij}\phi+\sum_i \hat{\beta}_i \D_i \phi +c \phi,
\end{aligned}
\end{equation}
with $\hat{\beta}_i=\beta_i+i\D_1a_{1i}$, $c=\D_1\beta_1$,  and the free term $f$ is given by
\begin{equation}                           \label{eq: free term}
f= (\D_1\beta_2) \D_2 u+ (\D_1 a_{22})\D_{22}u.
\end{equation}
For any $x \in  \overline{D}$, let  $\hat{X}^x =(\hat{Y}^{x_1}, \hat{Z}^x)$ be the unique solution of $\Pi(x;\hat{\beta},\sigma)$, where $\hat{\beta}=(\hat{\beta}_1, \hat{\beta}_2)$, and notice that $\D_1a_{11}(0) \geq 0$ (since $\sigma_1(0)=0$). Consequently, for all $t \in [0,T]$ we have that  $\hat X^x_t \in \ \overline{D}$.  
Let us set 
\begin{equation}            \label{eq: definition of v}
\begin{aligned}
v(t,x) :=& \E\left[ (\D_1g)(\hat{X}^{x}_{T-t})\exp\left(\int_0^{T-t} c(\hat{X}^x_s)\, ds\right)\right]
\\
 +& \E \left[ \int_0^{T-t} f(t+s, \hat{X}^{x}_s)\exp \left(\int_0^s c(\hat{X}^x_r)\, dr \right) \, ds  \right].
\end{aligned}
\end{equation}

To prove that $\D_1 u$  is continuous on $[0,T] \times \overline{D}$ we show that $v$ is continuous and that $v=\D_1 u$.

\begin{proposition}
Under Assumption \ref{as: 2}, the function $v$ defined in \eqref{eq: definition of v} is continuous on $[0,T] \times \overline{D}$. 
\end{proposition}

\begin{proof}
Let $(t^n,x^n) \in [0,T] \times \overline{D}$, $n \in \bN$,  converging to $(t,x)\in [0,T] \times \overline{D}$. By \eqref{eq: convergence in probability} of Lemma \ref{lem: convergence of diffusions} combined with the continuity and the boundedness of $\D_1g$ and $c$, we have
\begin{align*}
\lim_{n \to \infty} &\E\left[ (\D_1g)(\hat{X}^{x^n}_{T-t^n})\exp\left(\int_0^{T-t^n} c(\hat{X}^{x^n}_s)\, ds\right)\right]
\\
=&\E\left[ (\D_1g)(\hat{X}^{x}_{T-t})\exp\left(\int_0^{T-t} c(\hat{X}^x_s)\, ds\right)\right].
\end{align*}
Also, by \eqref{eq: convergence in probability} and  the continuity of $f$ and $c$ we have that
\begin{align*}
&I_{s< T-t^n}f(t^n+s, \hat{X}^{x^n}_s)\exp\left(\int_0^sc(\hat{X}^{x^n}_r) \,dr\right)  \\
\to & I_{s< T-t}f(t+s, \hat{X}^{x}_s)\exp\left(\int_0^sc(\hat{X}^{x}_r) \,dr\right),
\end{align*}
in measure (on $(0,T) \times \Omega$) as $n \to \infty$ (where we have set $f(r,\cdot)=f(T,\cdot)$ for $r >T$). The result now follows by the boundedness of $c$ and $f$ (recall \eqref{as2: regularity} from Assumption \ref{as: 2}). 
\end{proof}

For the proof of the next proposition we will need to define some approximation functions. Let $\zeta^n_i \in C^\infty(\bR)$ with $0 \leq \zeta^n_i \leq 1$ such that
\begin{enumerate}
\item $\zeta^n_1= 1$ on $(-\infty,1/(2n)]$, $\zeta^n_1= 0$ on $(1/n, \infty]$, $\zeta^n_1 >0$ on $(1/(2n), 1/n)$, and $|\D \zeta^n_1|\leq N n$ 

\item $\zeta^n_2= 1-\zeta^n_1$ on $(-\infty, n]$, $\zeta^n_2= 0$ on $[n^4, \infty)$, $\zeta^n_2 >0$ on $(n,n^4)$, and $|\D \zeta^n_2|\leq n^{-4}$ on $(n, n^4)$

\item  $\zeta^n_3= 0$ on $(-\infty,n]$, $\zeta^n_3= 1
$ on $(n^4, \infty]$, $\zeta^n_3 >0$ on $(n, n^4)$, and $|\D \zeta^n_3
|\leq N n^{-4} $ 
\end{enumerate} 

For $i \in \{1,2\}$ 
let us extend $\sigma_i$ and $\beta_i$ on $\bR$ by setting $\sigma_i(r)=\sigma_i(0)$ and $\beta_i(r)=\beta_i(0)$ for $r <0$, and
let us  set 
\[\sigma^n_i(r):=\left[2a_{ii}\left(\frac{1}{2n}\right) \zeta^n_1(r)+ 2a_{ii}(r) \zeta^n_2(r)+ \zeta^n_3(r)\right]^{1/2} ,\]
\[\beta_i^n(r)= \beta_i(\frac{1}{2n}) \zeta_1^n(r)+ \beta_i(r)\zeta_2^n(r),\] 
and as usual
$$
a^n_{ij}:= \lambda_{ij} \frac{\sigma^n_i \sigma^n_j}{2}.
$$

\begin{proposition}            \label{pr: continuous derivative x1}
Under Assumption \ref{as: 2},  we have $\D_1 u \in C([0,T] \times \overline{D})$. 
\end{proposition}
\begin{proof}
For $n \in \bN_+ $ let $\sigma^n_i$ and $\beta^n_i$ be the functions defined above. Under Assumption \ref{as: 2} it is not difficult to see that $\sigma_i^n, \beta_i^n \in  C_b^\infty (\bR)$, $\inf_{\bR} \sigma^n_i >0$, and the following hold:
\begin{enumerate}[i]
\item[(i)] $\sigma^n_i=\sigma_i$ and $\beta_i^n=\beta_i$ on $[1/n,n]$,

\item[(ii)] $\beta^n_i \to \beta_i$ and $\sigma^n_i \to \sigma_i$ uniformly on compacts subsets of $\bR$ as $n \to \infty$,
and there exists a constant $N$ such that $|\beta_1^n(r)|+|\sigma_1^n(r)| \leq N(1+|r|)$, for all $n \in \bN_+$, $r\in \bR$

\item[(iii)]  $\left( a_{22}^n\right)'$, $(\beta_1^n)'$, and  $(\beta_2^n)'$ are bounded, uniformly in $n \in \bN_+$.

\end{enumerate}
Let us set $\beta^n:=(\beta^n_1, \beta^n_2)$, $\sigma^n:=(\sigma^n_1, \sigma^n_2)$, and for every $n \in \bN_+$ let $L_n$  denote the generator of $\Pi(\cdot; \beta^n, \sigma^n)$. For every $n \in \bN_+$,  the equation 
\begin{equation} 
  \left\{ \begin{array}{ll}
         \D_t u^n  + L_n u^n=0  & \mbox{in $(0,T)  \times \bR^2$};\\
        u^n(T,x)=g(x) &\mbox{for $x \in \bR^2$}
        \end{array} \right. 
        \end{equation}
has a unique solution $u^n \in \mathbb{W}$ which moreover belongs to $\mathscr{W}$, and by the Feynman-Kac formula we have for all $(t,x) \in [0,T] \times \bR^2$
\begin{equation}                  \label{eq: F-K}
u^n(t,x)= \mathbb{E} g ( X^{x;n}_{T-t}), 
\end{equation}
where $X^{x;n}$ is the unique solution of $\Pi(x; \beta^n, \sigma^n)$. 
Let  $K$ be a compact subset of $(0,T) \times D$ and notice that on $K$,  for all $n\in \bN_+$ large enough, it holds that  $L_n=L$. Moreover, $L$ is strongly elliptic on $K$ and its coefficients and all their derivatives are bounded.  By virtue of Proposition \ref{pr: C 1 2}, for  any  $Q \subset \subset \text{int}(K)$, we obtain by standard parabolic estimates
\begin{align}                                \label{eq: energy}
\nonumber
\|\nabla u^n - \nabla u \|^2_{L_2(Q)} & \leq  N \|u^n -u \|^2_{L_2(K)} +N \|(L_n-L)u \|^2_{L_2(K)} \\
& = N \|u^n -u \|^2_{L_2(K)} ,
\end{align}
for $n$ large enough, with a constant $N$ independent of $n \in \bN$.  By the properties of $\beta^n$ and $ \sigma^n$,   Lemma \ref{lem: convergence of diffusions} and \eqref{eq: F-K} we have that $u^n(t,x) \to u(t,x)$ for all  $(t,x) \in K$, and since $|u^n(t,x)| \leq \|g\|_{L_\infty}< \infty$ for all $(t,x) \in K$, $n \in \bN_+$,  we get that   $\lim_{n \to \infty} \|u^n -u \|^2_{L_2(K)} = 0$,  which due to \eqref{eq: energy} implies that 
\begin{equation}       \label{eq: gradient convergence}
\lim_{n \to \infty} \|\nabla u^n - \nabla u \|^2_{L_2(Q)} =0.
\end{equation}
On the other hand, by differentiating $u^n$ with respect to $x_1$  we easily see that 
 $v^{n}:=\D_1 u^{n}$  belongs to  $\mathcal{W}$ and satisfies
  \begin{equation*} 
  \left\{ \begin{array}{ll}
         \D_tv^n  + \hat{L}_n v^n + f^n=0  & \mbox{on $(0,T) \times \bR^2$ }\\
        v^n(T)=\D_1g  & \mbox{on $\bR^2$, }
        \end{array} \right.
        \end{equation*}
 where 
 $$
\hat{L}_n \phi:=\sum_{ij} a_{ij}^n\D_{ij} \phi+ \sum_i \hat{\beta}^n_i \D_i\phi+ c^n \phi
$$
with
$$
\hat{\beta}^n_1= \beta^n_1 +\partial_1
a_{11}^n, \  \hat{\beta}^n_2= \beta^n_2 +2\partial_1 a_{12}^n, \ c^n=\partial_1 \beta^n_1
$$
  and 
$$
f^n=(\partial_1 \beta^n_2) \D_2 u^n+(\partial_1 a_{22}^n)\D_{22}u^n.
$$
By the Feynman-Kac formula we have
\begin{align} 
\nonumber
v^n(t,x) =& \E\left[ (\D_1g)(\hat{X}^{x;n}_{T-t})e^{\int_0^{T-t} c^n(\hat{X}^{x;n}_s) ds}\right]
\\
 +& \E \left[ \int_0^{T-t} f^n(t+s, \hat{X}^{x;n}_s)e^{\int_0^s c^n(\hat{X}^{x;n}_r) dr}  \, ds\right],
 \end{align}
 where $\hat{X}^{x;n}=(\hat{Y}^{x_1;n}, \hat{Z}^{x;n})$ is the unique solution of $\Pi(x;\hat{\beta}^n, \sigma^n)$. Let us set
 $$
 \tau_n := \inf\{t \geq 0 \colon Y^{x_1}_t  \not\in  (1/n,n)\}\wedge T, 
 $$ 
 and notice that on $[\![ 0, \tau_n  ]\!]:= \{ (\omega, t) \in \Omega \times [0,T] \colon t \leq \tau_n(\omega) \}$  both $\hat{X}^x$ and $\hat{X}^{x;n}$ satisfy $\Pi(x,\hat{\beta}^n, \sigma^n)$ and since $\hat{\beta}^n, \sigma^n$ are Lipschitz continuous we have that for all $n \in \bN_+$, $\hat{X}^{x;n}= \hat{X}^x$ on $[\![ 0, \tau_n  ]\!] $. In addition, by virtue of  \eqref{as2: exit} of Assumption \ref{as: 2} we have that zero is not an exit boundary for the diffusion $\hat{Y}^{x_1}$ (see, e.g.,  \cite[pp. 14]{MR1912205}). That is, if $x_1>0$,  then  $\inf_{t \in [0,T]} \hat{Y}^{x_1}_t >0$,   which in turn implies that almost surely  $\tau_n=T$ for $n$ sufficiently large.  In particular, for almost all $\omega  \in \Omega$, 
$ \sup_{t \in [0,T]} |\hat{X}^x_t-\hat{X}^{x;n}_t|=0$ for $n$ large enough depending on $\omega \in \Omega$.  Then notice that for each $(t,x)\in [0,T] \times D$, by the properties of $\sigma^n $ and $\beta^n$ we have
$$
\lim_{n \to \infty} \E\left[ (\D_1g)(\hat{X}^{x;n}_{T-t})e^{\int_0^{T-t} c^n(\hat{X}^{x;n}_s) ds}\right]=\E\left[ (\D_1g)(\hat{X}^{x}_{T-t})e^{\int_0^{T-t} c(\hat{X}^x_s) ds}\right].
$$
Moreover, notice that similarly to Proposition \ref{pr: representation derivatives} we have 
$\D_2u^n(t,x)= \E \D_2 g(X^{x;n}_{T-t})$ and $\D_{22}u^n(t,x)= \E \D_{22} g(X^{x;n}_{T-t})$, which by  virtue of (ii) above and Lemma \ref{lem: convergence of diffusions} implies that for any sequence $(x^n)_{n=1}^\infty  \subset D$ with $\lim_{n \to \infty} x^n=x \in D$, we have
$\lim_{n \to \infty} \D_2 u^n(t,x^n)=\D_2 u(t,x)$, and $\lim_{n \to \infty} \D_{22} u^n(t,x^n)=\D_{22} u(t,x)$. This combined with the properties of $\sigma^n$, $\beta^n$ imply in turn that $\lim_{n \to \infty} f^n(t,x^n) =f(t,x)$ whenever 
$(x^n)_{n=1}^\infty  \subset D$ with $\lim_{n \to \infty} x^n=x\in D$. In addition, $f^n$ are bounded in $(t,x) \in [0,T] \times \overline{D}$, uniformly in $n \in \bN$, and then one can easily see that for each $(t,x) \in [0,T] \times D$,
\begin{align*}
\lim_{n \to \infty}& \E \left[ \int_0^{T-t} f^n(t+s, \hat{X}^{x;n}_s)e^{\int_0^s c^n(\hat{X}^{x;n}_r) dr} \,ds \right] \\
=&\E \left[ \int_0^{T-t} f(t+s, \hat{X}^x_s)e^{\int_0^s c(\hat{X}^x_r) dr} \,ds \right].
\end{align*}
Consequently, for every $(t,x) \in [0,T] \times D$, we have $\lim_{n \to \infty} v^n(t,x) =v(t,x)$, which combined with \eqref{eq: gradient convergence} gives that $\D_1 u = v$ on $[0,T] \times D$. Since $v \in C([0,T] \times \overline{D})$, it follows that $u$ is differentiable with respect to $x_1$ on 
$[0,T] \times \overline{D}$ and $\D_1u =v \in C([0,T] \times \overline{D})$. This finishes the proof.
\end{proof}

Let $D$   be an open bounded domain in $\bR^d$ where $d \in \bN_+$, and for $i,j \in \{1,\dots, d\}$,  let $a^{ij},b^i,c \colon [0,T] \times D \to \bR$ be measurable functions. Let us set $\mathcal{L} \phi:= \sum_{i,j}\D_i(a^{ij}\D_{j} \phi)+\sum_i b^i\D_i+c\phi$ and $Q:= (0,T) \times D$. 

\begin{assumption}                                \label{as: Moser}
The functions $a^{ij}, \D_la^{ij}, b^i, c$  are bounded in magnitude by a constant $N_1$. Moreover there exists a constant $\varkappa >0$ such that   
$a^{ij}\xi_i \xi_j \geq \varkappa |\xi|^2$, 
for all $\xi=(\xi_1, \dots, \xi_d) \in \bR^d$. 
\end{assumption}

The following is well-known in the theory of parabolic PDEs and for a proof we refer the reader to  \cite[pp. 211, Theorem 11.1]{MR0241821}.

\begin{lemma}                                    \label{lem: Moser}
Suppose that Assumption \ref{as: Moser} holds and let $f \in L_2((Q)$ and $ Q' \subset Q$ such that $\text{dist}(Q', \D_pQ)>0$. Then there exists a constant $N$ depending only on $N_1$, $\varkappa$, $Q$, and $\text{dist}(Q', \D_pQ)$, such that  for any $u \in C^{1,2}(Q)$ satisfying $\D_t u = \mathcal{L} u +f$ on $Q$, the estimate
$$
\|\nabla u \|_{L_\infty(Q')} \leq N( \| u \|_{L_\infty(Q)} +\| f\|_{L_\infty(Q)})
$$
holds.
\end{lemma}

\begin{proposition}                \label{pr: blow up rate second derivative}
Under Assumption \ref{as: 2},  for each $i \in \{1,2\}$ we have 
$$
\lim_{n \to \infty}  a_{i1} (x^n)\D_{i1}u(t^n,x^n) =0,
$$
for any sequence  $(t^n,x^n)_{n=1}^\infty \subset (0,T) \times D$ with $\lim_{n \to \infty}(t^n,x^n)=(t^0,x^0)$, where   $x^0 \in \D D$, $t^0 \in (0,T)$.  
\end{proposition}

\begin{proof}
Let $\{(t^n,x^n)\}_{n=1}^\infty \subset (0,T) \times D$ be a sequence converging to $(t^0, x^0) \in (0,T) \times \D D$  and set $s^n=T-t^n$. Recall that $2a_{11}$ is Lipschitz near zero, with a Lipschitz constant $K>0$. Around the point $(t^n,x^n)$ consider  the rectangle
$$
R^n:= \left(s^n-\theta^n,s^n+\theta^n\right)\times (x_1^n-\theta^n, x_1^n+\theta^n) \times (x_2^n-\theta^n, x_2^n+\theta^n),
$$
where $\theta^n= a_{11}(x_1^n)/K$, 
and let $p_n \colon \bR^3 \to \bR^3$ be given by 
$$
p_n(t,x_1,x_2)= \left(\frac{Tt}{4\theta^n}+T\left(\frac{3}{4}-\frac{s^n}{4\theta^n}\right), \frac{x_1+2\theta^n-x_1^n}{2\theta^n},   \frac{x_2-x_2^n}{2\theta ^n}\right).
$$
It follows that $p_n(R^n)=(T/2,T) \times (1/2,3/2) \times (-1/2,1/2)=:Q$, and $p_n(s^n, x_1^n,x_2^n)= (3T/4,1,0)=:q$,   for all $n \in \bN$. 
Since $u \in C^{1,2}((0,T)\times D)$ and satisfies $\D_t u+Lu=0$, we have for 
$v(t,x):=u(T-t,x)$ that $v \in C^{1,2}((0,T)\times D)$ and satisfies $\D_t v=Lv$ on $(0,T) \times D$. Moreover, since the coefficients are smooth in $D$ it follows that $v \in C^{1,k} (Q)$ for any $k \in \bN$ (see, e.g. Theorem 10, page 72 in \cite{MR0181836}); in particular,  $\hat{v}:=\D_1 v \in C^{1,2}(R^n)$.  It is easy to see that  $\D_t\hat{v} = \hat{L} \hat{v}+f$ on $R^n$, where $\hat{L}$ and $f$ are given  in \eqref{eq: operator hat} and \eqref{eq: free term} respectively.  It follows then that for all $n \in \bN$, $w^n:= \hat{v}\circ p_n^{-1} \in C^{1,2}(Q)$ and it satisfies on $Q$
$$
\D_tw^n= \hat{L}_nw^n+f^n,
$$
where 
\begin{equation}                  \label{eq: operator hat n}
\begin{aligned}
\hat{L}_n \phi := 
\sum_{ij} \frac{a_{ij}\circ p_n^{-1}}{T\theta^n} \D_{ij}\phi+ \sum_i \frac{2\hat{\beta}_i\circ p_n^{-1}}{T} \D_i \phi +\frac{4\theta^n}{T} c\circ p_n^{-1}\phi
\end{aligned}
\end{equation}
and $f^n:=(4\theta^n/T) f \circ p^{-1}_n$. Recall that $a_{11}$ is Lipschitz continuous near zero with Lipschitz constant $K/2$. Consequently, for any $x_1 \in (1/2, 3/2)$ we have 
$$
a_{11}(2\theta^n(x_1-1)+x_1^n) \geq - K\theta^n |1-x_1| + a_{11}(x_1^n),
$$
which implies that on $Q$ we have 
$$
\frac{a_{11}\circ p_n^{-1}}{\theta^n }\geq - K |1-x_1| +K \geq \frac{K}{2}.
$$ 
By \eqref{as2: diffusion behaviour} of Assumption \ref{as: 2} we have on $Q$ (for all  $n$ sufficiently large)
$$
\frac{a_{22}\circ p_n^{-1}}{\theta^n }\geq \frac{N_0 K}{2}.
$$
Moreover, one can easily check that the coefficients of $\hat{L}_n$ are bounded on $Q$ uniformly in $n \in \bN$. Consequently the operators  $\hat{L}_n$ satisfy the assumption of Lemma~\ref{lem: Moser} with constants $N_1$ and $\varkappa$ independent of $n \in \bN$. 
 Let $Q' \subset Q$ be a cylinder with $\text{dist}(Q',\D_pQ)>0$ and such that $q \in Q'$, and set $\varrho:= \hat{v}(s^0,x^0)$. Notice that on $Q$ we have 
 $$
 \D_t(w^n-\varrho) =\hat{L}_n(w^n-\varrho)+ \frac{4\theta^n}{T} \varrho c \circ p_n^{-1}+f^n, 
 $$
and by virtue of  Lemma \ref{lem: Moser}  there exists a constant $N$ such that for all $n \in \bN$ large enough
$$
\|\nabla w^n\|_{L_\infty(Q')} \leq N (\|w^n -\varrho\|_{L_\infty(Q)}+\theta^n  \| c \circ p_n^{-1}\|_{L_\infty(Q)}+\|f^n\|_{L_\infty(Q)}).
$$
Since $f$ and $c$ are  bounded we have 
$$
\lim_{n \to \infty}\|f^n\|_{L_\infty(Q)}=\lim_{n \to \infty}\theta^n\|c \circ p^{-1}_n\|_{L_\infty(Q)}=0.
$$
 Also, since $\hat{v} \in C([0,T] \times \overline{D})$ we have
$$
\lim_{n \to \infty}\|w^n -\varrho\|_{L_\infty(Q)}= \lim_{n \to \infty} \| \hat{v} -\varrho\|_{L_\infty(R^n)}=0.
$$
Consequently, for each $i \in \{1,2\}$ we have
\begin{align*}
|a_{i1}(x^n_1)\D_{i1}u(t^n,x^n)|&=|a_{i1}(x^n_1)\D_i \hat{v}(s^n,x^n)|\\
& \leq N \theta^n|\D_i \hat{v}(s^n,x^n)|  \\
& \leq N |\D_iw^n (q)| \leq N \|\nabla w^n\|_{L_\infty(Q')} \to 0,
\end{align*}
where the first inequality above follows from (\ref{as2: diffusion behaviour}) of Assumption \ref{as: 2}. This brings the proof to an end.
\end{proof}

\begin{proposition}                 \label{prop: continuity time derivative}
Under Assumption \ref{as: 2} we have that $\D_t u \in C( (0,T) \times \overline{D})$.
\end{proposition}

\begin{proof}
We show first that $u(t,x)$ is differentiable with respect to $t \in (0,T)$ for any $x \in \D D$. For $t \in (0,T)$ and $x \in \D D$ we have 
$$
\frac{u(t+h,x)-u(t,x)}{h} = \frac{u(t+h,x_{(h)})-u(t,x_{(h)})}{h}+ O(h)
$$
where $x_{(h)}=(h^2, x_2)$. By the mean value theorem we have that the right hand side of the above inequality is equal to 
$$
(\D_tu)(t+\xi(h), x_{(h)})+O(h)= -(Lu)(t+\xi(h), x_{(h)})+O(h)
$$
for some $\xi(h) \in [0,h]$. Consequently, by virtue of Propositions \ref{pr: representation derivatives}, \ref{pr: continuous derivative x1}, and \ref{pr: blow up rate second derivative} we obtain
$$
\lim_{h \to 0}\frac{u(t+h,x)-u(t,x)}{h} = -a_{22}(0)\D_{22}u(t,x)-\sum_i \beta_i(0)\D_iu(t,x).
$$
Hence, the time derivative exists. Moreover, by the above equality combined again with Propositions \ref{pr: representation derivatives}, \ref{pr: continuous derivative x1}, and \ref{pr: blow up rate second derivative} and the fact that $\D_tu=-Lu$ on $(0,T) \times D$, it follows that $\D_tu$ is continuous on $(0,T) \times \overline{D}$.

\end{proof}

\begin{proof}[Proof of Theorem \ref{thm: solution of PDE}]
The fact that $u$ is indeed a solution follows from Propositions \ref{prop: continuity u} to \ref{prop: continuity time derivative}. Hence we proceed with the uniqueness part.
It suffices to show  that if $g=0$ and $u$ is a solution of \eqref{eq: parabolic pde} having polynomial growth, then $u\geq 0$. To this end, let $v$ be a solution of \eqref{eq:parabolic pde} such that for some constant $N$ we have for all $(t,x) \in (0,T) \times D$ that $|v(t,x)| \leq N(1+|x|^{m-1})$ with an integer $m \geq 2$. Let us also set $\tilde{v}(t,x)= v(T-t,x)$.
 Let $w(t,x)=1+|x_1|^m+|x_2|^m$ and notice that due to \eqref{as2: regularity} of Assumption \ref{as: 2}  we obtain that 
$ Lw < cw-1$ for a sufficiently large constant $c$. Let us set $\tilde{v}^\varepsilon= u+\varepsilon e^{ct}w$ and notice that on $(0,T) \times D$ we have 
\begin{equation}                           \label{eq: strictly superparabolic}
\D_t \tilde{v}^\varepsilon-L \tilde{v}^\varepsilon= \varepsilon e^{ct}(cw-Lw)>\varepsilon e^{ct}.
\end{equation}
Assume that $\{\tilde{v}^\varepsilon <0 \}=:\Gamma$ is non-empty for some $\varepsilon>0$ (otherwise there is nothing to prove)  and notice that by the growth condition on $v$ and the definition of $w$ we have that $\Gamma$ is bounded. Let $s= \inf \Gamma_{[0,T]}$, where $\Gamma_{[0,T]}$ is the projection of $\Gamma$ on $[0,T]$. Since $\overline{\Gamma}$ is compact, there exists $ z=(z_1,z_2) \in \overline{D}$ such that $(s,z) \in \overline{\Gamma}$ and by the continuity of $\tilde{v}^\varepsilon$ we get $\tilde{v}^\varepsilon(s,z)=0$ which in particular implies that $s> 0$. First assume that $ z \in \D D$. By definition of $s$ we have that 
$\tilde{v}^\varepsilon (t,z) \geq 0$ for $t <s$, $\tilde{v}^\varepsilon (s,0,x_2) \geq 0$ for all $x_2 \in \bR$, and $\tilde{v}^\varepsilon (s,x_1,z_2) \geq 0$ for all $x_1 \geq 0$. Consequently, 
$\D_t\tilde{v}^\varepsilon(s,z) \leq 0$, $\D_2\tilde{v}^\varepsilon(s,z)=0$, $\D_{22}\tilde{v}^\varepsilon(s,z) \geq 0$, and $\D_1\tilde{v}^\varepsilon(s,z)\geq 0$. Since $\beta_1(0), a_{11}(0)\geq 0$ we obtain
\begin{equation}      \label{eq: boundary epsilon}
\delta:=\D_t \tilde{v}^\varepsilon(s,z)- \beta_1(0)\D_1\tilde{v}^\varepsilon(s,z)-\beta_2(0)\D_2\tilde{v}^\varepsilon(s,z)-a_{22}(0)\D_{22}\tilde{v}^\varepsilon(s,z)  \leq 0 
\end{equation}
It follows from Definition \ref{def: solution PDE}  that 
$$
\lim_{(0,T) \times D \ni (t,x) \to (s,z)} a_{11}\D_{11}\tilde{v}+a_{12}\D_{12}\tilde{v}=0
$$
which in turn implies that
$$
\lim_{(0,T) \times D \ni (t,x) \to (s,z)} a_{11}\D_{11}\tilde{v}^\varepsilon+a_{12}\D_{12}\tilde{v}^\varepsilon=0,
$$
where we have used that $a_{12}(0)=a_{11}(0)=0$.
Combining this with \eqref{eq: strictly superparabolic} and \eqref{eq: boundary epsilon} gives 
$$
\varepsilon e^{cs} \leq \lim_{(0,T) \times D \ni (t,x) \to (s,z)}  \D_t \tilde{v}^\varepsilon-L \tilde{v}^\varepsilon = \delta \leq 0,
$$
which is a contradiction. Hence, $z_1>0$ and $z\in D$ is a local minimum of the function $\tilde{v}^\varepsilon(s,\cdot)$. It follows then that $L \tilde{v}^\varepsilon(s,z) \geq 0$ or else, if $L \tilde{v}^\varepsilon(s,z)<0$, then on a ball around $z$ we have $L \tilde{v}^\varepsilon(s,z)<0$ and since   $\tilde{v}^\varepsilon(s, \cdot)$ has minimum on $z$ we have by the Hopf maximum principle (see, e.g.,  \cite[pp. 349, Theorem 3]{MR2597943}) that $\tilde{v}^\varepsilon(s,\cdot)=0$ near $z$ and in particular $L\tilde{v}^\varepsilon(s,z)=0$. Hence 
$$
\D_t \tilde{v}^\varepsilon(s,z)-L \tilde{v}^\varepsilon(s,z) \leq 0,
$$
which again contradicts \eqref{eq: strictly superparabolic}. This shows that $\tilde{v}^\varepsilon \geq 0$, and since this is true for all $\varepsilon>0$ we have $v \geq 0$. This brings the proof to an end. 
\end{proof}

\bibliographystyle{plain}
\bibliography{BibKE}

\end{document}